\documentclass[12pt,leqno]{article}% ,draft

\usepackage{amsfonts}
\usepackage{amsmath}
\usepackage{amssymb}
\usepackage{amscd}
\newcommand{\rarrow}[1]{{\buildrel #1 \over \longrightarrow}}

\makeatletter
  \@addtoreset{equation}{section}
\makeatother
\makeatletter

\newcommand{\Z}{\mathbb{Z}}

\newtheorem{thm}{Theorem}[section]

\newtheorem{prop}[thm]{Proposition}

\newtheorem{lem}[thm]{Lemma}

\newtheorem{Asser}[thm]{Assertion}

\newenvironment{proof}{\noindent{\bf Proof.}}
{\noindent \ \hfill$\Box$\par}

\begin{document}

\title{On Mimura's extension problem}

\author{Toshiyuki Miyauchi, Juno Mukai, Mariko Ohara \\}
\date{}
\maketitle

\noindent
{\bf Abstract} We determine the group strucure of the $23$-rd homotopy
group $\pi_{23}(G_2 : 2)$, where $G_2$ is the Lie group of exceptional type, which hasn't been determined for $50$
years. 
\footnote{2010 Mathematics Subject
Classification. Primary 57T20; Secondary 55R10.\\
Key words and phrases. Lie group of exceptional type, $\Z_2$-Moore
space, extension, whitehead product, Toda bracket.}

\section{Introduction}
Let $X$ be a space. Let $\pi_k(X:2)$ denote the $2$-primary or infinite
components
of the $k$-th homotopy group of $X$, i.e., its index $[\pi_k(X): \pi_k(X:2)]$ is odd.
In 1967, Mimura \cite[Theorem 9.1]{Mi2} showed that there are isomorphisms
\[
\pi_{23}(Spin(7);2)\cong\pi_{23}(G_2;2)\oplus \pi_{23}(S^7:2), 
\]
\[
\pi_{23}(Spin(9);2)\cong\pi_{23}(G_2;2)\oplus \pi_{23}(S^9:2),  
\]
where $G_2$ is the Lie
group of exceptional type. However, the structure of the $23$-rd homotopy
group $\pi_{23}(G_2 : 2)$ hasn't been determined completely :   

\begin{Asser}[Mimura's problem~\cite{Mi2}, cf. \cite{DP}; p.369, cf. \cite{Lu}]\label{mimuprob}
\[
 \pi_{23}(G_2:2)\cong \Z_4 \oplus \Z_2 \quad \text{or} \quad (\Z_2)^3 .
\] 
\end{Asser}
The aim of this paper is to determine the structure of $\pi_{23}(G_2 :
2)$. 

To explain our strategy, we will state Assertion~\ref{mimuprob} more
explicitely along \cite{Mi2}. 
%Let $SO(n)$ be the rotation group of degree $n$. 
Let $SU(n)$ be the $n$-th special unitary group. %$r_n: SU(n)\to SO(2n)$ the canonical inclusion.
Let $SU(3) \rarrow{i_G} G_2 \rarrow{p_G}S^6$ and $SU(2)\rarrow{i_U}SU(3)\rarrow{p_U}S^5$ be the canonical fiber sequences. 
Let $\langle \beta \rangle \in \pi_k(G_2 : 2)$ be an element satisfying $p_{G
\ast}\langle \beta \rangle =\beta$ for $\beta \in \pi_k(S^6 : 2)$, and
$[\alpha] \in \pi_k(SU(3) : 2)$ an element satisfying $p_{U \ast}[\alpha]
 =\alpha$ for $\alpha \in \pi_k(S^5 :2)$. 
By using these sequences, Mimura obtained the explicit elements that
generate $\pi_{23}(G_2:2)$ as follows. 

\begin{prop}[\cite{Mi2}]\label{G20}
Let $G$ be an abelian group spanned by the elements $\langle
 P(E\theta)+\nu_6\kappa_9\rangle$ and
 $i_{G\ast}[\nu_5\bar{\varepsilon}_8]$. Then, 
$\pi_{23}(G_2:2)
%=\{\langle
 %P(E\theta)+\nu_6\kappa_9\rangle,i_{G\ast}[\nu_5\bar{\varepsilon}_8],\langle\eta_6\mu_7\rangle\sigma_{16}\}
\cong G\oplus\Z_2\{\langle\eta_6\mu_7\rangle\sigma_{16}\}$,
where $G=\Z_4$ or $(\Z_2)^2$. If $2\langle
 P(E\theta)+\nu_6\kappa_9\rangle=i_{G \ast}[\nu_5\bar{\varepsilon}_8]$, then $G=\Z_4$. 
\end{prop}
Mimura constructed the lift $\langle P(E\theta)+\nu_6\kappa_9 \rangle$
 by using the equation $\Delta (P(E\theta))$ $=$ $ \Delta
 (\nu_6\kappa_9)$. 
We determine that $G=\Z_4$ by showing the relation $2\langle
 P(E\theta)+\nu_6\kappa_9\rangle=i_{G \ast}[\nu_5\bar{\varepsilon}_8]$.  
Let $M^n$ be the Moore space $M^n=S^{n-1} \cup_{2\iota_{n-1}} e^n$ of type $(\Z_2,n-1)$, and $[M^n, S^k]$ the
cohomotopy group of a Moore space. 
Our essential idea is to find the Toda brackets represented by
 extensions of $P(E\theta)$ and $\nu_6\kappa_9$ over $M^{24}$ (Lemmas 3.1
 and 3.2),
 respectively. The crucial consequence of our method is Theorem 4.8. By
 using this theorem, we have the following. 
\begin{thm}\label{pi23G_2}
$2\langle P(E\theta)+\nu_6\kappa_9\rangle=i_{G\ast}[\nu_5\bar{\varepsilon}_8]
 \in \pi_{23}(G_2:2)$. 
\end{thm}

We use the notations
and results of \cite{Mi2} and \cite{T} freely. We write $\pi_k^n$ for
$\pi_{k}(S^n :2)$. 
%We denote by $\Delta_{G_2}$ and $\Delta_U$ the connecting maps of the
%long exact sequences of the homotopy groups obtained by the cofiber
%sequence $G_2 \to G_2 / SU(3)=S^6$ and $SU(3) \to SU(3)/SU(2)=S^5$
%respectively.
We also denote by $\Delta: [M^{24}, S^6] \to
[M^{23}, SU(3)]$ 
%the connecting map of cohomotopy groups of Moore spaces
 the connecting map induced by the fibration $G_2 \to G_2 / SU(3) =S^6$. 
%Note that we have the relation
%$\Delta(\bar\alpha)\circ i'_{24}=\Delta_{G_2}(\alpha)$, where $\bar\alpha
%\in [M^{24}, S^6]$ is an extension of an order $2$ element $\alpha \in \pi_{23}^6$ over
%$M^{24}$. 
 
For the element
$\kappa_7$, we adopt the renamed one by the equation $2\kappa_7 =
\bar\nu_7\nu^2_{15}$ in ~\cite[(15.5)]{MT1}. 
For a group $H$ and $\alpha\in H$, we denote by $\sharp\alpha$ the order of $\alpha$. 
\section{Some cohomotopy groups of $M^n$}
We have a cofiber sequence
\begin{equation}\label{cohomotopy}
\cdots\rarrow{p'_{n}}S^{n}\rarrow{2\iota_{n}}S^{n}
\rarrow{i'_{n+1}}M^{n+1}\rarrow{p'_{n+1}}S^{n+1}
\rarrow{2\iota_{n+1}}S^{n+1}\rarrow{i'_{n+2}}\cdots,
\end{equation}
where $i'_n:S^{n-1}\hookrightarrow M^n$ is the inclusion and 
$p'_n: M^n\to S^n$ is the collapsing map. 

By using the exact sequence induced from (\ref{cohomotopy}) 
%$S^{n+5}\rarrow{2\iota_{n+5}}S^{n+5}\rarrow{i_{n+6}}M^{n+6}$
 and the fact that $\pi^n_{n+1}=\{\eta_n\} \cong \Z_2$ for $n \ge 3$,  
% and the homotopy group structure of spheres,  
%$\pi^{19}_{k}$ and $\pi^{14}_{k}$ for $k=20$
% \cite[Propositions 5.1, 5.9 and 5.11]{T},
we obtain
\begin{equation}\label{M20}
[M^{n+1},S^{n}]=\Z_2\{\eta_{n}p'_{n+1}\}
\quad\text{for}\quad n \ge 3 .
\end{equation}
By using the exact sequence induced from (\ref{cohomotopy}) 
and the fact that $\pi^n_{n+6}=\{\nu^2_n\} \cong \Z_2$ for $n \ge 5$, we obtain 
\begin{equation}\label{M201}
[M^{n+6},S^{n}]=\Z_2\{\nu^2_{n}p'_{n+6}\} \quad\text{for}\quad n \ge 6 .
\end{equation}

Let $\bar{\eta}_n\in[M^{n+2},S^n]$ and \mbox{$\tilde{\eta}_n\in\pi_{n+2}(M^{n+1})$} 
for $n\ge 3$ be an extension and a coextension 
of $\eta_n$, respectively.
We note that
\[
\bar{\eta}_n\in\{\eta_n,2\iota_{n+1},p'_{n+1}\}\ \ \text{and}\ \ 
\tilde{\eta}_n\in\{i'_{n+1},2\iota_n,\eta_n\}, 
\;\mbox{for}\; n\ge 3.
\]
We have
$$
2\bar{\eta}_n=\eta^2_np'_{n+2}\;\;\mbox{and}\;\;
2\tilde{\eta}_n=i'_{n+1}\eta^2_n,\; \mbox{for}\; 
n\ge 3
$$
and
\begin{equation}\label{hoco}
[M^{n+2}, S^n]=\Z_4\{\bar\eta_n\} \quad \text{and} \quad
 \pi_{n+2}(M^{n+1})=\Z_4\{\tilde{\eta}_n\} \quad \text{for} \quad n
 \ge 3 .
\end{equation}
Since the indeterminacy of $\{\eta_n, 2\iota_{n+1}, p'_{n+1}\}$ is
$\eta_n \circ [M^{n+2}, S^{n+1}]+\pi^n_{n+2} \circ
p'_{n+2}=\{\eta^2_np'_{n+2}\}$ by \cite[Proposition 5.3]{T} and (\ref{M20}), we obtain 
\begin{equation}\label{eet}
\{\eta_n, 2\iota_{n+1}, p'_{n+1}\}=\bar\eta_n + 2\{\bar\eta_n\} \quad
 \text{for} \quad n \ge 3. 
\end{equation}
By the definition of $\nu'$ \cite[p.40]{T} and \cite[Proposition 1.7, Lemma 5.4]{T} that
\begin{equation}\label{bntn}
%\bar{\eta}_n\tilde{\eta}_{n+1}=\pm E^{n-3}\nu'
\bar\eta_3 \tilde{\eta}_4 = \pm \nu' 
\quad \text{and} \quad 
\bar{\eta}_{n}\tilde{\eta}_{n+1}=\pm 2\nu_{n} \quad
 \text{for} \quad n \ge 5. 
\end{equation}
Let $\iota_{M^n}$ be the homotopy class of the identity map of 
$M^n$. 
In view of \cite[p. 307, Corollary]{T1}, it holds
\begin{equation} \label{ietap}
2\iota_{M^n}=i'_n\eta_{n-1}p'_n, \;\mbox{for}\; n\ge 4.
\end{equation}
%By \cite[(3.4) Corollary]{Sp62}, we have
%\begin{equation}\label{relSpanier}
%\iota_{n}\in\{2\iota_{n},p'_{n},i'_{n}\} \quad\text{for\ } n\ge 2.
%\end{equation}
Note that $[M^n, M^n]=\Z_4\{\iota_{M^n}\}$ for $n \ge 4$. 

By the relations $\nu_6\eta_9=0$ \cite[(5.9)]{T} and $\{\nu_6, \eta_9,
2\iota_{10}\}_1 = P(\iota_{13})+2\{P(\iota_{13})\}$ ~\cite[Lemma
5.10]{T}, we have 
\[
\nu_6\bar\eta_9 \in \nu_6 \circ \{\eta_9, 2\iota_{10},
p'_{10}\}_1 = \{\nu_6, \eta_9, 2\iota_{10}\} \circ p'_{11} = P(p'_{13})
\]
that is, 
%$\nu_6\bar\eta_{9}=P(p'_{13})$, so we obtain $\nu_7\bar{\eta}_{10}=0$.  
we have the relations 
\begin{equation}\label{j3}
\nu_6\bar\eta_9 =P(p'_{13}), \quad \nu_n\bar\eta_{n+3}=0 \quad (n \ge
 7). 
\end{equation}

Let $\overline{\nu^2}_5 \in [M^{12}, S^5]$ be an extension of $\nu^2_5$. 
We set $\overline{\nu^2}_{n}=E^{n-5}\overline{\nu^2}_{5}$ for $n\ge 5$.
Then, we have the following : 
\begin{lem}\label{21}
\begin{itemize}
\item[(1)] $\sharp (\overline{\nu^2}_n) =2$ for $n \ge 5$
\item[(2)] $\{\eta_n, \nu_{n+1}, \overline{\eta}_{n+4}\}_1 = \overline{\nu^2}_n$ for $n \ge7$
\item[(3)] $\eta_n\overline{\nu^2}_{n+1}=\varepsilon_np'_{n+8}$
 for $n \ge 5$.
\end{itemize}
\end{lem}
\begin{proof}
(1) By (\ref{ietap}), we have
 $2\overline{\nu^2}_5=\overline{\nu^2}_5 \circ
 2\iota_{M^{12}}=\overline{\nu^2}_5 \circ i_{12}
 \eta_{11}p'_{12}=\nu^2_5\eta_{11}p'_{12}=0$. 

(2) Recall the relations $\eta_5\nu_6=0$~\cite[(5.9)]{T} and $\nu^2_6=\{\eta_6,\nu_7,\eta_{10}\}$~\cite[Lemma 5.12]{T}. 
We have
\begin{align*}
\overline{\nu^2}_{6} \circ i'_{13}=\nu^2_6
=\{\eta_6,\nu_7,\eta_{10}\}
&= \{\eta_6,\nu_7,\bar\eta_{10}\circ i'_{12}\}
\supset\{\eta_6,\nu_7,\bar\eta_{10}\}\circ i'_{13}\\
&\bmod \eta_6 \circ \pi^7_{12}+\pi_{11}^6 \circ \eta_{11} . 
\end{align*}
Since $\pi^7_{12}=0$, we have $\eta_6 \circ \pi^7_{12}=0$. 
From the fact that $\pi^6_{11}=\{P(\iota_{13})\} \cong \Z$~\cite[Proposition 5.9]{T} 
%and $\bar{\eta}_{11} \circ i'_{13}=\eta_{11}$
, we have $\pi^6_{11} \circ \eta_{11} =\{P(\eta_{13})\}$. 
Since $E: \pi^6_{12} \to \pi^7_{13}$ is an
 isomorphism~\cite[Proposition 5.11]{T}, we obtain $P(\eta_{13})=0$. 
Hence we have $\overline{\nu^2}_6 \circ i'_{13}= \{\eta_6, \nu_7,
 \bar{\eta}_{10}\} \circ i'_{13}$. By using
 $\pi^6_{13}=\Z_4\{\sigma''\}$~\cite[Proposition 5.15]{T}, we obtain
 that $\{\eta_6, \nu_7, \overline{\eta}_{10}\} \equiv \overline{\nu^2}_6 \bmod \sigma^{''}p'_{13}$ .  
This implies $\{\eta_7, \nu_{10}, \bar\eta_{13}\}_1 = E\{\eta_6, \nu_9,
 \bar\eta_{12}\} \equiv E\overline{\nu^2}_6 =\overline{\nu^2}_7 \bmod E(\sigma''p'_{13}) =2\sigma'\circ
 p'_{14}=0$ because $E\sigma''=2\sigma'$~\cite[Lemma 5.14]{T}. 
%the first implies $\{\eta_7, \nu_{10},
% \eta_{13}\}_1 = E\{\eta_6, \nu_9, \eta_{12}\} \equiv
% E\overline{\nu^2}_6=\overline{\nu^2}_7 \bmod E(\sigma''p'_{13})=2\sigma'\circ
% p'_{14}=0$. This leads to the second.  
%We have $2(\sigma''p'_{13})=\sigma'' \circ 2p'_{13}=0$. 
%Since $\{2\iota_5, \eta_5, \nu_6\} \subset \pi^5_{10}=\{\nu_5\eta^2_8\} \cong
% \Z_2$, we see that $2\iota_6 \circ \{\eta_6, \nu_7, \eta_{10}\} =
% -\{2\iota_6, \eta_6, \nu_7\} \circ \bar\eta_{11} \supset E\{2\iota_5,
% \eta_5, \nu_6\} \circ \bar\eta_{11}=0 \bmod 2\iota_6 \circ \pi^6_{11}
% \circ \bar\eta_{11} + \pi^6_8 \circ \nu_8\bar\eta_{11}=0$ by the
% relation $[2\iota_6, 2\iota_6] \circ \bar\eta_{11}=[\iota_6, \iota_6]
% \circ 4\iota_{11}$. 

(3) We know that $\pi^5_{13}=\{\varepsilon_5\} \cong \Z_2$ and so, $\varepsilon_5=\{\eta_5, \nu^2_6, 2\iota_{12}\}$
~\cite[(6.1) Theorem 7.1]{T}. We take $\overline{\nu^2}_6 \in \{\nu^2_6,
 2\iota_{12}, p'_{12}\}$. So, from the relation $\eta_5\sigma''
 =0$~\cite[(7.4)]{T}, 
%Since $\overline{\nu^2}_6 \in \{\nu^2_6, 2\iota_{12}, p'_{12}\}$, 
we see that $\eta_5\overline{\nu^2}_6 \in \eta_5 \circ \{\nu^2_6, 2\iota_{12},
 p'_{12}\}=\{\eta_5, \nu^2_6, 2\iota_{12}\} \circ
 p'_{13}=\varepsilon_5p'_{13} \bmod \eta_5 \circ \pi^6_{13} \circ
 p'_{13}=\{\eta_5\sigma''p'_{13}\}=0$. 
This leads to the assertion. 
\end{proof}

\begin{lem}\label{mato}
$\varepsilon_5\overline{\nu^2}_{13}=\bar{\varepsilon}_5p'_{20}$
\end{lem}
\begin{proof}
Since $\bar{\varepsilon}_5=\{\varepsilon_5,\nu^2_{13},2\iota_{19}\}_1$
 by \cite[III-Proposition 2.3 (5); the second]{Od2} and
 $\varepsilon_n\sigma_{n+8}=0$ for $n \ge 3$~\cite[Lemma
 10.7]{T}, Lemma~\ref{21} induces 
\begin{align*}
\bar\varepsilon_5p'_{20}=\{\varepsilon_5, \nu^2_{13}, 2\iota_{19}\}_1
 \circ p'_{20} 
=\varepsilon_5 \circ E\{\nu^2_{12},2\iota_{18},p'_{18}\}
\ni \varepsilon_5\overline{\nu^2}_{13} \\
\bmod \quad \varepsilon_5 \circ E\pi^{12}_{19} \circ p'_{20} 
=\{\varepsilon_5\sigma_{13}\}p'_{20}=0. 
\end{align*}
\end{proof}
Recall the relation $\nu_5\zeta_8=\sigma'''\sigma_{12}$ \cite[Lemma 2.3]{KM2}. 
Let $\overline{\sigma'''\sigma_{12}}$ be an extension of $\sigma'''\sigma_{12}$.
\begin{lem}\label{M22}
$[M^{22},S^7]=\{E^2\overline{\sigma'''\sigma_{12}},\bar\nu_7\overline{\nu^2}_{15}, \rho''p'_{22},\sigma'\bar\nu_{14}p'_{22},
\sigma'\varepsilon_{14}p'_{22},\bar\varepsilon_{7}p'_{22}\}
\cong(\Z_2)^6$.
\end{lem}
\begin{proof}
Recall that $\pi^7_{21}=\{\sigma'\sigma_{14}, \kappa_7\} \cong \Z_8
 \oplus \Z_4$~\cite[Theorem 10.3]{T} and $\pi^7_{22}=\{\rho'', \sigma'\bar{\nu}_{14}, \sigma'\varepsilon_{14}, \bar{\varepsilon}_7\} \cong \Z_8 \oplus
 (\Z_2)^3$~\cite[Theorem 10.5]{T}, where
 $4(\sigma'\sigma_{14})=E^2(\sigma'''\sigma_{12})$ by \cite[Lemma 5.14]{T} and $2\kappa_7=\bar{\nu}_7\nu^2_{15}$. 
So, by using (\ref{cohomotopy}) for $n=21$, we have 
\[
 [M^{22}, S^7]=\{E^2\overline{\sigma'''\sigma_{12}}, \bar{\nu}_7\overline{\nu^2}_{15}\}+ \pi^7_{22} \circ p'_{22}.
\]
We know from \cite[(7.1), (7.4)]{T} : 
\begin{equation}\label{j5}
\eta_9\sigma_{10}+\sigma_9\eta_{16}=P(\iota_{19}) \quad \text{and} \quad 
 \sigma_{10}\eta_{17}=\eta_{10}\sigma_{11} . 
\end{equation}
By (\ref{ietap}), (\ref{j5}) and \cite[(7,4)]{T}, we obtain 
$\;2\overline{\sigma'''\sigma_{12}}
=\sigma'''\sigma_{12} \circ
\eta_{19}p'_{20}=(\sigma'''\eta_{12})\sigma_{13}p'_{20}$ $=0$. 
By Lemma~\ref{21} (1),  
%the result that $\overline{\nu^2}_{15}=E^{10}\overline{\nu^2}_5$, 
we have $2(\bar{\nu}_7 \overline{\nu^2}_{15})=0$.
This completes the proof.
\end{proof}

\section{Toda brackets for extensions over $M^{24}$ of $P(E\theta)$ and $\nu_6\kappa_9$}
First of all, we recall ~\cite[(7.21) ; the third]{T}
\begin{equation}\label{rn}
\sigma_{11}\nu_{18}=P(\iota_{23}), \quad \text{so that} \quad
 \sigma_{n}\nu_{n+7}=0 \quad \text{for} \quad n \ge 12. 
\end{equation}
Hence, we have the well-defined element $\theta \in \{\sigma_{12}, \nu_{19},
\eta_{22}\}_1$~ \cite[pp.73--74]{T}, whose indeterminacy 
%of $\{\sigma_{12}, \nu_{19}, \eta_{22}\}_1$ 
is $\sigma_{12} \circ E\pi^{18}_{23} + \pi^{12}_{23} \circ \eta_{23} =\{P(\eta_{25}), \zeta_{12}\eta_{23}\}$. 

We know $\zeta_6\eta_{17}=8P(\sigma_{13})$ and $\zeta_{n}\eta_{n+2}=0$
for $n \ge 7$ ~\cite[Proposition 2.2 (6)]{Og}.
We also know $P(\eta_{25})=E\theta'$~\cite[(7.30) ; the second]{T}. 
Therefore, $\{\sigma_{12}, \nu_{19}, \eta_{22}\}_1 = \theta + \{E\theta'\}$.

Since $\sigma_{12} \circ \pi^{19}_{24}=0$ and $\sigma_{12} \circ E^{13}\pi^6_{11}=0$, we have
 $\{\sigma_{12},\nu_{19},\eta_{22}\}_1=\{\sigma_{12},\nu_{19},\eta_{22}\}_n$
 $(0 \le n \le 13)$.
So, we take $\{\sigma_{12},\nu_{19},\eta_{22}\}_3 = \theta +
\{E\theta'\}$, and $E\theta=\{\sigma_{13},\nu_{20},\eta_{23}\}_4$ from
the fact $\pi^{13}_{24} \circ \eta_{24}=\{\zeta_{13}\eta_{24}\}=0$.

We recall the following equation from \cite[Proposition 2.2 (2)]{Og}  
\begin{equation}\label{j6}
\mu_n\eta_{n+9}=\eta_n\mu_{n+1} \quad \text{for} \quad n \ge 3.
\end{equation}
We have that $P(E\theta) \in \{P(\sigma_{13}),\nu_{18},\eta_{21}\}_2
 \, \bmod \, \pi^6_{22} \circ \eta_{22}$. 
Note that $\pi^6_{22}=\{\zeta_6',
\mu_6\sigma_{15}, \eta_6\bar\varepsilon_7 \} \cong \Z_8 \oplus
(\Z_2)^2$~\cite[Theorem 12.6]{T}. 
We know $\zeta'\eta_{22}=0$~\cite[Proposition 2.13 (5)]{Og} and  
$\eta_6\bar{\varepsilon}_7\eta_{22}=(\nu_6\sigma_9\nu^2_{16})\circ
\eta_{22}=0$ by \cite[Lemma 12.10]{T}. 
By (\ref{j5}) and (\ref{j6}), we have the relation
$\mu_5\sigma_{14}\eta_{21} =\mu_5\eta_{14}\sigma_{15} =
\eta_5\mu_6\sigma_{15}$.  
Hence, we obtain $\pi^6_{22} \circ \eta_{22}
 =\{\eta_6\mu_7\sigma_{16}\}$ and $P(E\theta) \in
 \{P(\sigma_{13}),\nu_{18},\eta_{21}\}_2 \bmod \eta_6\mu_7\sigma_{16}$. 

We consider the Toda bracket
 $\{P(\sigma_{13}),\nu_{18},\bar{\eta}_{21}\}_2 \subset [M^{24}, S^6]$. 
By the relation (\ref{rn}), %$\sigma_{13}\nu_{20}=0$, and (\ref{M201}), 
we have $P(\sigma_{13}) \circ \nu_{18}=P(\sigma_{13}\nu_{21})=0$. So, by
 (\ref{M201}), we obtain 
\begin{equation}\label{13}
P(\sigma_{13}) \circ E^{11}[M^{13}, S^7]
=P(\sigma_{13}) \circ [M^{24}, S^{18}]
%=P(\sigma_{13}) \circ \{\nu^2_{18}p'_{22}\}
%=\{P(\sigma_{13}\nu_{20}) \circ \nu_{21}p'_{23}\}
=0 .
\end{equation}
By (\ref{j3}) and (\ref{13}), we have $\{P(\sigma_{13}),\nu_{18},\bar{\eta}_{21}\}=\{P(\sigma_{13}),\nu_{18},\bar{\eta}_{21}\}_n$
 for $(0 \le n \le 11)$. 

%We recall the definition of $\theta$ in $\pi^{12}_{24}$ : $\theta \in
%\{\sigma_{12}, \nu_{19}, \eta_{22}\}_1$, from \cite{T}. Since
%$\sigma_{12} \circ E\pi^{18}_{23}= \sigma_{12} \circ E^3 \pi^{16}_{20}=0$,
%we can take $\theta \in \{\sigma_{12}, \nu_{19}, \eta_{22}\}_3$. 
Let $\overline{P(E\theta)} \in [M^{24}, S^6]$ be an extension of
 $P(E\theta)$. Then, we show : 
\begin{lem}\label{san}
$\overline{P(E\theta)} \in \{P(\sigma_{13}),\nu_{18},\bar{\eta}_{21}\}_2 \bmod \pi^6_{22} \circ \bar\eta_{22} + \pi^6_{24} \circ p'_{24}$
\end{lem}
\begin{proof}
Notice that $\overline{P(E\theta)}$ is a representative of the Toda
 bracket $\{P(E\theta), 2\iota_{23}, p'_{23}\}_2$. 
We use the Jacobi identity \cite[Proposition 1.5]{T} :  
\begin{align*}
\{ \{P(\sigma_{13}),\nu_{18},{\eta}_{21}\}, 2\iota_{23}, p'_{23}\} +
 \{P(\sigma_{13}), \{\nu_{18}, \eta_{21}, 2\iota_{22}\}, p'_{23}\} \\ +
 \{P(\sigma_{13}),\nu_{18}, \{\eta_{22}, 2\iota_{22},
 p'_{22}\}\}\equiv 0.  
\end{align*}
%Since $\{\eta_{22}, 2\iota_{22},
% p'_{22}\}=\bar\eta_{21} + \{2\bar\eta_{21}\}$
By (\ref{eet}), we rewrite the identity:
\begin{align*}
\{ \{P(\sigma_{13}),\nu_{18},{\eta}_{21}\}, 2\iota_{23}, p'_{23}\} +
 \{P(\sigma_{13}), \{\nu_{18}, \eta_{21}, 2\iota_{22}\}, p'_{23}\} \\ +
 \{P(\sigma_{13}),\nu_{18}, x\bar{\eta}_{21}\} \equiv 0 \quad \text{for} \quad x :
 \text{odd}.  
\end{align*}
%Here, we use the fact that $\{\eta_{22}, 2\iota_{22},
% p'_{22}\}=\{x\bar\eta_{21}\}$. 

From the fact that $\{\nu_{18}, \eta_{21}, 2\iota_{22}\} \subset
 \pi^{18}_{23}=0$ and (\ref{13}),
 the second term is 
$\{P(\sigma_{13}), 0, p'_{23}\}=P(\sigma_{13})\circ [M^{24},
 S^{18}]+\pi^6_{24}\circ p'_{24}=\pi^6_{24}\circ p'_{24}$. 
%since $\{\nu_{18}, \eta_{21}, 2\iota_{22}\}$ is a quotient subset of
 %$\pi^{18}_{23} =0$. 
The indeterminacy of the first term is 
\begin{align*}
\{P(\sigma_{13}),\nu_{18},{\eta}_{21}\} \circ [M^{24}, S^{23}] +
 \pi^6_{24} \circ p'_{24}.
\end{align*}
%By $\zeta7\eta_{22}=0$~\cite[Proposition 2.13 (5)]{Og},
% $\eta_6\bar\epsilon_7 = \nu_6\sigma_9\nu^2_{16}$~\cite[Lemma 12.10]{T}
By (\ref{M20}), we have $\{P(\sigma_{13}), \nu_{18}, \eta_{21}\} \circ
 [M^{24}, S^{23}] \subset \{P(\sigma_{13}), \nu_{18}, \eta_{21}\} \circ
 \{\eta_{23}p'_{24}\}$ $\subset$ $\pi^6_{24} \circ p'_{24}$. 
%and the relation $\mu_6\sigma_{15}\eta^2_{22}=4\zeta_6\sigma_{17}$, we have
% $\pi_{22}^6 \circ \eta^2_{22} \circ p'_{24}=0$. 
So, the indeterminacy is $\pi_{24}^6 \circ p'_{24}$

The indeterminacy of the third term is
\begin{align*}
P(\sigma_{13}) \circ [M^{24}, S^{18}] + \pi_{22}^6 \circ \bar\eta_{22} =
 \pi_{22}^6 \circ \bar\eta_{22}
\end{align*}
%We have $P(\sigma_{13}) \circ [M^{24}, S^{18}] =0$
 by (\ref{13}). 
Thus the assertion is proved. 
\end{proof}
Recall that $\eta_7\sigma_8 =(\sigma'\eta_{14}+\bar\nu_7
+\varepsilon_7)$ by \cite[(7.4)]{T}. 
We have the relation 
\[
 (\ast) \quad \eta_7\sigma_8\nu^2_{15}= (\bar{\nu}_7+\varepsilon_7
 +\sigma'\eta^2_{14})\nu^2_{15}=\bar{\nu}_7\nu^2_{15}=2\kappa_7 .
\] 
Let us take an extension $\overline{\nu_6\kappa_9} \in \{\nu_6\kappa_9,
2\iota_{23}, p'_{23}\}_2$ over $M^{24}$ of $\nu_6\kappa_9$. 

By the relations
$\varepsilon_3\nu_{11}=\nu'\bar\nu_6=0$~\cite[(7.12)]{T}, $E^2\nu'=2\nu_5$ and $(\ast)$, we have 
\begin{align*}
\eta_7\sigma_{8}\nu^2_{15}p'_{21}
%&=(\sigma'\eta_{14}+\bar{\nu}_7+\varepsilon_7)\nu^2_{15}p'_{21}
%=\bar{\nu}_7\nu^2_{15}p'_{21}
=2\kappa_7p'_{21}=\kappa_7\circ 2p'_{21}=0.
\end{align*}
% by using \cite[Lemma
%6.4 (7.4) (5.9) (7.12)]{T}.
Therefore the Toda bracket 
$\{\nu_6,\eta_9,\sigma_{10}\nu^2_{17}p'_{23}\}_2$ 
is well-defined.

We have $\overline{\nu_6\kappa_9} \in \{\nu_6\kappa_9,
2\iota_{23}, p'_{23}\}_2 \subset \{\nu_6, 2\kappa_9, p'_{23}\}_2 \supset
\{\nu_6, \eta_{9}, \sigma_{10}\nu^2_{17}p'_{23}\}_2$ since
$2\kappa_9=\bar\nu_9\nu^2_{17}=\eta_9\sigma_{10}\nu^2_{17}$. The
indeterminacy of the second bracket is $\nu_6 \circ E^2[M^{22},
S^7] + \pi^{6}_{24} \circ p'_{24}$. This implies 
\begin{lem}\label{322}
$\overline{\nu_6\kappa_9} \in
	   \{\nu_6,\eta_9,\sigma_{10}\nu^2_{17}p'_{23}\}_2 \bmod \nu_6
	   \circ E^2[M^{22}, S^7] + \pi^6_{24} \circ p'_{24}$.
\end{lem}
%Considering $\eta_2 \wedge \bar\varepsilon_3$
By \cite[Lemma 12.10]{T}, we have 
\begin{equation}\label{j7}
\bar\varepsilon_n\eta_{n+15}=\eta_n\bar\varepsilon_{n+1} \quad
 \text{for} \quad n \ge 3.
\end{equation}
We know from \cite[(7.5)]{T}
\begin{equation}\label{j8}
\varepsilon_n\eta_{n+8}=\eta_n\varepsilon_{n+1} \quad
 \text{for} \quad n \ge 3.
\end{equation}
We recall the equations from \cite[Lemma 6.3]{T} that 
\begin{equation}\label{hs}
\eta_n\bar\nu_{n+1} = \bar\nu_n\eta_{n+8} = \nu^3_n \quad
 \text{for} \quad n \ge 6.
\end{equation}
Next, we show :
\begin{lem}\label{lA}
$\{2\iota_7, \eta_7, \sigma_8\nu^2_{15}p'_{21}\} =
 \bar{\nu_7}\overline{\nu^2}_{15}% \bmod \{\rho'', \sigma'\bar\nu_{14},
 % \sigma'\varepsilon_{14}\} \circ p'_{22}
$
\end{lem}
\begin{proof}
%We have the relation $\eta_7\sigma_8\nu^2_{15}= (\bar{\nu}_7+\varepsilon_7
% +\sigma'\eta^2_{14})\nu^2_{15}=\bar{\nu}_7\nu^2_{15}=2\kappa_7$, so
 % that 
By $(\ast)$, we have $\{2\iota_7, \eta_7, \sigma_8\nu^2_{15}p'_{21}\}$
 $\subset$ $\{2\iota_7, 2\kappa_7, p'_{21}\}\supset
 \{\bar{\nu}_7\nu^2_{15}, 2\iota_{21}, p'_{21}\} \ni
 \bar\nu_7\overline{\nu^2}_{15} \bmod 2\iota_7 \circ [M^{22}, S^7]+\pi^7_{22}
 \circ p'_{22}$. By \cite[Theorem 10.5]{T}, $\pi^7_{22}
 \circ p'_{22} = \{\rho'',
 \sigma'\bar{\nu}_{14}, \sigma'\varepsilon_{14}, \bar\varepsilon_7\}
 \circ p'_{22}$. 
By the fact that $S^{7}$ is an $H$-space and by Lemma~\ref{M22}, we have $2\iota_7 \circ [M^{22}, S^7]=2[M^{22},
 S^7]=0$. Hence, we obtain  
\[
 \{2\iota_7, \eta_7, \sigma_8\nu^2_{15}p'_{21}\} \ni
 \bar{\nu}_7\overline{\nu^2}_{15} \bmod \{\rho'', \sigma'\bar{\nu}_{14},
 \sigma'\varepsilon_{14}, \bar\varepsilon_7\} \circ p'_{22}.
\]
We set 
\[
(\ast \ast) \quad \{2\iota_7, \eta_7,
 \sigma_8\nu^2_{15} p'_{21}\} = \bar{\nu}_7\overline{\nu^2}_{15} + (a\rho'' +b \sigma'\bar{\nu}_{14} +c
 \sigma'\varepsilon_{14},+ d\bar\varepsilon_7) \circ p'_{22}, 
\]
where $a, b, c, d \in \{0, 1\}$. 
We compose $\tilde{\eta}_{21}$ to $(\ast \ast)$ on the right. We know
 the relations $\rho''\eta_{22}=\sigma'\mu_{14}$ \cite[Proposition
 2.8 (3) ; the second]{Og},
 $\sigma'\bar{\nu}_7\eta_{15}=\sigma'\nu^3_{14}= \nu_7\sigma_{10}\nu^2_{17}=\eta_7\bar{\varepsilon}_8=\bar{\varepsilon}_7\eta_{22}$
 by (\ref{hs}), (\ref{j7}) and \cite[(7.19), Lemma 12.11]{T}. Moreover,
 we know $\sigma'\varepsilon_{14}\eta_{22}=\sigma'\eta_{14}\varepsilon_{15}=E\zeta'$
 by \cite[(12.4)]{T} and (\ref{j8}). 
This implies
$(a\rho'' +b \sigma'\bar{\nu}_{14} +c
 \sigma'\varepsilon_{14},+ d\bar\varepsilon_7) \circ p'_{22} \circ
 \tilde{\eta}_{21}
 =a\sigma'\mu_{14}+(b+d)\eta_7\bar{\varepsilon}_8+cE\zeta'$. 

On the other hand, we have
$\{2\iota_7, \eta_7,
 \sigma_8\nu^2_{15} p'_{21}\} \circ \tilde{\eta}_{21}
=2\iota_7 \circ \{\eta_7,\sigma_8\nu^2_{15}p'_{21},\tilde{\eta}_{20}\}
\subset 2\iota_7 \circ \pi^7_{23}=2\pi^7_{23}=0$~\cite[Theorem 12.6]{T}.

We know that $\varepsilon_5 \in \{\nu_5^2, 2\iota_{11}, \eta_{11} \}$~\cite[(7.6)]{T}. 
So, by the relation
 $\bar{\nu}_6\varepsilon_{14}=0$~\cite[Proposition 2.8 (2)]{Og} and $\bar{\nu}_6\sigma_{14}=0$~\cite[Lemma 10.7]{T},  
we have 
\[
 \bar\nu_7\overline{\nu^2}_{15} \circ \tilde{\eta}_{21}
\in \bar\nu_7 \circ \{\nu^2_{15},2\iota_{21},\eta_{21}\}
= \bar\nu_7(\varepsilon_{15} + \pi^{15}_{22} \circ \eta_{22})
= \bar\nu_7\varepsilon_{15} + \bar\nu_7 \circ \{\sigma_{15}\eta_{22}\} =0.
\]
So, $(\ast \ast)$ becomes 
$0=a\sigma'\mu_{14}+(b+d)\eta_7\bar{\varepsilon}_8+cE\zeta'$ and hence, 
we have $a=c=(b+d)=0$ by seeing $\pi^7_{23}=\{\sigma'\mu_{14}, E\zeta',
 \mu_7\sigma_{16}, \eta_7\bar\varepsilon_8\}$~\cite[Theorem 12.6]{T}.
Therefore we obtain
\[
 \{2\iota_7, \eta_7,
 \sigma_8\nu^2_{15} p'_{21}\} = \bar{\nu}_7\overline{\nu^2}_{15} + \{b \sigma'\bar{\nu}_{14} + d\bar\varepsilon_7\} \circ p'_{22}
\]
We notice that $\bar\nu_6^2=0$~\cite[Proposition 2.8 (2)]{Og} 
%$\eta_6\rho''=4\zeta'$ by \cite{Og}
 and $\eta_6\sigma'=4\bar\nu_6$ by \cite[(7.4)]{T}. 
So, by composing $\eta_6$ on the left for the above equation, we have
 $\eta_6 \circ \{2\iota_7, \eta_7, \sigma_8\nu^2_{15}p'_{21}\} =
 \eta_6\bar{\nu}_7\overline{\nu^2}_{15} +d \eta_6\bar\varepsilon_7p'_{22}$. 

By (\ref{bntn}), 
we have 
\[
 \eta_6 \circ \{2\iota_7, \eta_7, \sigma_8\nu^2_{15}p'_{21}\} =
 \{\eta_6, 2\iota_7, \eta_7\} \circ \sigma_9\nu^2_{16}p'_{22} = \pm
 2\nu_6 \circ \sigma_9\nu^2_{16}p'_{22}=0 .
\]
By the relations $\bar\nu_6\varepsilon_{14}=0$ from \cite[Proposition 2.8
 (2)]{Og}, (\ref{hs}) %$\eta_6\bar\nu_7=\bar\nu_6\eta_{14}$~\cite[Lemma 6.3]{T}
 and Lemma~\ref{21}(3), we see that  
\[
 \eta_6\bar\nu_7\overline{\nu^2}_{15} =\bar\nu_6\eta_{14}\overline{\nu^2}_{15}= \bar\nu_6\varepsilon_{14}p'_{22}=0 .
%\bmod \bar\nu_6\eta_{14}\sigma_{15}p'_{22}=0. 
\]
Since $\eta_6\bar\varepsilon_7$ generates a direct summand $\Z_2$ in
 $\pi^6_{22}$ by \cite[Theorem 12.6]{T}, we have $\eta_6\bar\varepsilon_7p'_{22} \neq
 0$. This implies $d=0$ and completes the proof.
\end{proof}

\section{Proof of the main theorem}
%Recall that $\Delta: [M^{24}, S^6] \to [M^{23}, SU(3)]$ is the map induced by
%the fibration $G_2 \to G_2 / SU(3) =S^6$. 
We recall the equation from \cite[(4.4)]{MT2}  
\begin{equation}\label{j1}
%2\iota_5 \circ \nu_5\sigma_8=2\nu_5\sigma_8 \quad \text{and} \quad 
2\iota_5 \circ \zeta_5\sigma_{16}=2\zeta_5\sigma_{16} .
\end{equation}

We show :
\begin{lem}\label{31}
\begin{itemize}
\item[(1)] ${p_U}_*\Delta(\pi^6_{22} \circ
 \bar{\eta}_{22})=0$ and
 ${p_U}_*\Delta(\pi^6_{24} \circ p'_{24})=0$.
\item[(2)] 
$p_{U\ast}\Delta(\overline{P(E\theta)}) =
	   p_{U\ast}\Delta \{P(\sigma_{13}),\nu_{18},\bar{\eta}_{21}\}_2$
\end{itemize}
\end{lem}
\begin{proof}
By Lemma~\ref{san}, it suffices to show the assertion (1). 
By \cite[Proposition 6.3]{Mi2}, we have 
$\Delta(\zeta')=\Delta(\mu_6\sigma_{15})={i_U}_*\mu'\sigma_{14}$, 
$\Delta(\eta_6\bar{\mu}_7)={i_U}_*\nu'\bar{\mu}_6$,
 $\Delta(\eta_6\bar\varepsilon_7)=0$, $\Delta(P(E\theta) \circ
 \eta_{23})=0$ and $\Delta(\zeta_6\sigma_{17})=[2\iota_5]\zeta_5\sigma_{16}$. 
So, we have ${p_U}_*\Delta(\pi^6_{22} \circ
 \bar{\eta}_{22})={p_U}_*\Delta(\pi^6_{22})\bar{\eta}_{21}=0$ and
 ${p_U}_*\Delta(\pi^6_{24} \circ
 p'_{24})={p_U}_*\Delta(\pi^6_{24})\circ p'_{23}
=\{2\iota_5 \circ \zeta_5\sigma_{16}p'_{23}\}
=\{2(\zeta_5\sigma_{16})p'_{23}\}=0$ by (\ref{j1}). 
%using the relation $2\iota_5 \circ \zeta_5\sigma_{16}=2\zeta_5\sigma_{16}$~\cite[(4.4)]{MT2}.
\end{proof}
Next, we show :
\begin{lem}\label{32}
$p_{U\ast}\Delta(\overline{\nu_6\kappa_9})=p_{U\ast}\Delta\{\nu_6, \eta_9, \sigma_{10}\nu^2_{17}p'_{23}\}_2$.
\end{lem}
\begin{proof}
By Lemma~\ref{322} and Lemma~\ref{31} (1), it suffices to show that $p_{U\ast}\Delta(\nu_6 \circ E^2[M^{22}, S^7])=0$. 
Since $\Delta(\nu_6 \circ E^2[M^{22}, S^7]) \subset \Delta\nu_6
 \circ E^2[M^{22}, S^7]$ and the relation
 $\Delta\nu_6=[2\iota_5]\nu_5$~\cite[Proposition 6.2]{Mi2},
 we have $p_{U\ast}\Delta(\nu_6 \circ E^2[M^{22}, S^7])=2\nu_5 \circ
 E[M^{22}, S^7]=0$. 
%By Lemma~\ref{31}(1), we have $p_{U\ast}\Delta(\pi_{24}^6 \circ p'_{23}) %\subset \{2\zeta_5\sigma_{16}\} \circ p'_{23}=0$. 
\end{proof}
We show the following.
\begin{lem}\label{42}
$p_{U\ast}\Delta(P(E\theta))=2\nu_5\kappa_8$ 
\end{lem}
\begin{proof}
By \cite[p.164]{Mi2}, we have 
\[
p_{U\ast}\Delta(P(E\theta))=2\nu_5\kappa_8 + a(\eta_5\mu_6\sigma_{15}),
\]
where $a \in \{0, 1\}$. Since
 $\Delta(\nu_6\kappa_9)=[2\iota_5]\nu_5\kappa_8 \in
 \pi_5(SU(3):2)$~\cite[Proposition 6.2]{Mi2}, we have
 $p_{U\ast}\Delta(\nu_6\kappa_9)=2\nu_5\kappa_8$. 
For the connecting map $\Delta_U : \pi^5_{22} \to
 \pi_{21}(SU(2):2)=\pi^3_{21}$ induced from the
 fibration $S^3=SU(2) \to SU(3) \to S^5$, we have $\Delta_U \circ
 p_{U\ast}=0$. 
By \cite[Proposition 3.2(i)]{MT2}, we have
 $\Delta_U(\eta_5)=\eta^2_3$. By the relation
 $2\mu'=\eta^2_3\mu_5$~\cite[(7.7)]{T}, we have 
\[
 0=a\Delta_U(\eta_5\mu_6\sigma_{15})=a(\eta^2_3\mu_5\sigma_{14})
=2a(\mu'\sigma_{14}). 
\]
By \cite[Theorem 12.8]{T}, $\sharp(\mu'\sigma_{14})=4$, so that we have $a=0$.
\end{proof}

Let $r_n: SU(n)\to SO(2n)$ be the canonical inclusion.
By %\cite[p. 210--1]{Y} and
\cite[Corollary 5.3, 5.4 Theorem 5.5]{A}, the fibrations
$G_2\rarrow{p_G}G_2/SU(3)=S^6$ and $SO(7) \to SO(7)/SO(6)=S^6$ 
give a commutative diagram
\begin{equation} \label{Y}
\begin{CD}
SU(3) @>{i_G}>> G_2 @>{p_G}>> S^6 \\
@V{r_3}VV @V{h}VV @V{=}VV \\
SO(6) @>{i_7}>> SO(7) @>{p_7}>> S^6.
\end{CD}
\end{equation} 
We denote by $\Delta_R$ the connecting map obtained by the fibration
$SO(7) \to SO(7)/SO(6)=S^6$. 

\begin{lem}\label{43}
$\Delta(P(E\theta))=\Delta(\nu_6\kappa_9)$. 
\end{lem}
\begin{proof}
In \cite{Mi2}, Mimura obtained the consequence of this lemma, because the extension problem is based on it. However, we rewrite it as this lemma since we could not find its proof.  

We know the relation $i_{U\ast}\bar\mu' =
 \Delta(\bar\mu_6)$~\cite[Proposition 6.3]{Mi2}. 
So, by the relation $\Delta(P(E\theta)) \equiv [2\iota_5]\nu_5\kappa_8 \bmod i_{U\ast}\bar\mu' =\Delta(\bar\mu_6)$
~\cite[p.164, 3-rd line]{Mi2}, we can set 
$\Delta(P(E\theta))+a\Delta(\bar{\mu}_6)=\Delta(\nu_6\kappa_9)$, where
 $a \in \{0, 1\}$. 
By applying $r_3: U(3) \to R_6$, we have 
$\Delta_R(P(E\theta)) + a\Delta_R(\bar\mu_6) = \Delta_R(\nu_6\kappa_9)$.  
We know that $\Delta_R(P(E\theta))= 2[\nu_5]\kappa_8 = \Delta_R(\nu_6\kappa_9)$ in
 $R^6_{22}$ by the proof of \cite[Lemma 3.5]{HKM}. 
This implies $a\Delta_R(\bar\mu_6) =0$. 
Since $\Delta_R(\bar\mu_6)=[\iota_3]_6\bar\mu' = 8[\nu_4\rho'']_6 \ne
 0$, we have $a=0$. 
\end{proof}

\begin{lem}\label{Fa1}
$\Delta(\overline{P(E\theta)}+\overline{\nu_6\kappa_9}) =
 a[\nu_5\bar\varepsilon_8]p'_{23} +
 b[2\iota_5]\zeta_5\sigma_{16}p'_{23}$, where $a \in \{0,1\}$ and $b \in
 \{0,1,2,3\}$
\end{lem}
\begin{proof}
By virtue of Lemma~\ref{43}, we have the equation :
$\Delta(\overline{P(E\theta)} + \overline{\nu_6\kappa_9}) \circ
 i'_{23} = %\Delta_{G_2}((\overline{P(E\theta)} +
% \overline{\nu_6\kappa_9})i'_{24})=
 \Delta(P(E\theta) + \nu_6\kappa_9)=0$.
From this equation and (\ref{cohomotopy}) for $n=23$, 
%cofibering $S^{22} \rarrow{i'_{23}} M^{23} \rarrow{p'_{23}} S^{23}$, 
we have $\Delta(\overline{P(E\theta)}+\overline{\nu_6\kappa_9}) \in
 \pi_{23}(SU(3):2) \circ p'_{23}$. 

By \cite[Theorem 4.1]{MT2},
 $\pi_{23}(SU(3):2)=\{[2\iota_5]\zeta_5\sigma_{16},
 [\nu_5\bar\varepsilon_8]\} \cong \Z_4 \oplus \Z_2$. 
This implies the desired relation. 
\end{proof}
We show :
\begin{lem}\label{l1}
${p_U}_*\Delta(\overline{P(E\theta)})
={p_U}_*\Delta\{P(\sigma_{13}),\nu_{17},\bar{\eta}_{20}\}_2
=\nu_5\bar{\nu}_8\overline{\nu^2}_{16} +
 \nu_5\bar{\varepsilon}_8p'_{23}$. 
\end{lem}
\begin{proof}
From the relation $\Delta P(\iota_{13})=[\nu_5\eta^2_8]$
 ~\cite[Corollary 5.3]{Mi2} and Lemma~\ref{31}(2), %and Lemma~\ref{21}, 
we have ${p_U}_*\Delta(\overline{P(E\theta)})
={p_U}_*\Delta\{P(\sigma_{13}),\nu_{18},\bar{\eta}_{21}\}_2 
\subset \{\nu_5\eta^2_8\sigma_{10},\nu_{17},\bar{\eta}_{20}\}_1$. 
By Lemma~\ref{21} and Lemma~\ref{mato}, we obtain 
\begin{align*}
%{p_U}_*\Delta(\overline{P(E\theta)})
%={p_U}_*\Delta\{P(\sigma_{13}),\nu_{18},\bar{\eta}_{21}\}_2 
%\subset 
\{\nu_5\eta^2_8\sigma_{10},\nu_{17},\bar{\eta}_{20}\}_1 
=\{\nu_5(\bar{\nu}_8+\varepsilon_8)\eta_{16},\nu_{17},\bar{\eta}_{20}\}_1 
\supset
 \nu_5(\bar{\nu}_8+\varepsilon_8)\{\eta_{16},\nu_{17},\bar{\eta}_{20}\}_1 \\
=\nu_5(\bar{\nu}_8+\varepsilon_8)\overline{\nu^2}_{16} =
 \nu_5\bar{\nu}_8 \overline{\nu^2}_{16} + \nu_5\varepsilon_8p'_{23} \\
\quad \bmod \quad \nu_5(\bar{\nu}_8+\varepsilon_8)\eta_{16} \circ E[M^{22}, S^{16}]
+\pi^5_{21} \circ \bar{\eta}_{21}=\pi^5_{21} \circ \bar{\eta}_{21} 
=\{\mu_5\sigma_{14}\bar{\eta}_{21}\}.
\end{align*}
Here, we use the relations $[M^{22}, S^{16}]=\{\nu^2_{16}p'_{22}\}$ ~(\ref{M201}) and $\eta_5\bar{\varepsilon}_6\bar{\eta}_{21}=\nu_5\sigma_8\nu^2_{15}\bar{\eta}_{21}=0$ by using \cite[Lemma 12.10]{T} and (\ref{j3}).%$\nu_7\bar\eta_{10}=0$. 
%By Lemma~\ref{mato}, we have
%$\nu_5(\bar{\nu}_8+\varepsilon_8)\overline{\nu^2}_{16}
%=\nu_5\bar{\nu}_8\overline{\nu^2}_{16} + \nu_5\bar{\varepsilon}_8p'_{23}$.

This leads to the relation. 
\[
{p_U}_*\Delta(\overline{P(E\theta)})=\nu_5\bar{\nu}_8\overline{\nu^2}_{16}
+ \nu_5\bar{\varepsilon}_8p'_{23} + x\mu_5\sigma_{14}\bar{\eta}_{21},
\]
where $x \in \{0, 1\}$.

Applying $i'_{23}$ to this equality on the right, then we have
${p_U}_*\Delta(P(E\theta))=2\nu_5\kappa_8
+ x\mu_5\sigma_{14}\eta_{21}
=2\nu_5\kappa_8+x\eta_5\mu_6\sigma_{15}$.
By Lemma~\ref{42}, 
we obtain ${p_U}_*\Delta(\overline{P(E\theta)})=2\nu_5\kappa_8$. This implies $x=0$ and completes the proof.
\end{proof}

\begin{lem}\label{Fa2}
$p_{U\ast}\Delta(\overline{\nu_6\kappa_9}) =p_{U\ast}\Delta\{\nu_6, \eta_9,
 \sigma_{10}\nu^2_{17}p'_{23} \}_2 =\nu_5\bar\nu_8\overline{\nu^2}_{16}$. 
\end{lem}
\begin{proof}
By Lemma~\ref{32} and Lemma~\ref{lA}, we have 
\[\begin{split}
p_{U\ast}\Delta(\overline{\nu_6\kappa_9}) = p_{U\ast}\Delta\{\nu_6, \eta_9,
 \sigma_{10}\nu^2_{17}p'_{23} \}_2 \subset \{p_{U\ast}\Delta\nu_6,
 \eta_8,  \sigma_{9}\nu^2_{16}p'_{22} \}_1 & \\
=\{2\nu_5, \eta_8,
 \sigma_{9}\nu^2_{16}p'_{22}\}_1 \supset \nu_5 \circ E\{2\iota_7,
 \eta_7, \sigma_8\nu^2_{15}p'_{21}\} = \nu_5\bar\nu_8
 \overline{\nu^2}_{16}. & \\
\bmod \quad 2\nu_5 \circ E[M^{22}, S^7]+\pi^5_{10} \circ
 \sigma_{10}\nu^2_{17}p'_{23}
\end{split}\]
Since $\pi^5_{10} \circ
 \sigma_{10}\nu^2_{17}=\{\nu_5\eta_8^2\sigma_{10}\nu^2_{17}\}=\{\nu_5(\bar\nu_6
 + \varepsilon_8)\eta_{16}\nu^2_{17}\}=0$, we have the equation 
$p_{U\ast}\Delta(\overline{\nu_6\kappa_9}) =\nu_5\bar\nu_8\overline{\nu^2}_{16} \bmod 2\nu_5 \circ E[M^{22}, S^7]+\pi^5_{10} \circ
 \sigma_{10}\nu^2_{17}p'_{23}=0$. 
\end{proof}

Finally, we show :
\begin{thm}\label{M1}
\begin{equation}\label{MM}
 \Delta(\overline{P(E\theta)}+\overline{\nu_6\kappa_9}) =
 [\nu_5\bar\varepsilon_8]p'_{23} +
 b[2\iota_5] \zeta_5 \sigma_{16} p'_{23},
\end{equation}
where $b \in \{0,1,2,3\}$. 
\end{thm}
\begin{proof}
Since the order of $\nu_5\bar\nu_8\overline{\nu^2}_{16}$ is $2$, in
Lemma~\ref{l1} and Lemma~\ref{Fa2}, we have 
\begin{equation}\label{Se1}
 {p_U}_*(\Delta
 \overline{P(E\theta)}+\Delta(\overline{\nu_6\kappa_9}))=\nu_5\bar{\varepsilon}_8p'_{23}.%=p_{U\ast}([\nu_5\bar\epsilon_8]p'_{23}),
\end{equation}
%where $p_U: SU(3) \to SU(3)/SU(2)=S^5$ is projection.
We know $p_{U\ast}([2\iota_5]\zeta_5\sigma_{16})p'_{23}=2(\zeta_5\sigma_{16})p'_{23}=0$
 by (\ref{j1}). So, Lemma~\ref{Fa1} induces  
\begin{equation}\label{Se}
{p_U}_*(\Delta
 (\overline{P(E\theta)}+\overline{\nu_6\kappa_9}))={p_U}_*(a[\nu_5\bar{\varepsilon}_8]p'_{23}
 +  b[2\iota_5]\zeta_5\sigma_{16}p'_{23})=a \nu_5\bar{\varepsilon}_8p'_{23},
\end{equation}
where $a \in \{0,1\}$ and $b \in \{0,1,2,3\}$. 

By (\ref{Se1}) and (\ref{Se}), we have $\nu_5\bar\varepsilon_8p'_{23} = a
 \nu_5\bar\varepsilon_8p'_{23}$. 
Since $\nu_5\bar{\varepsilon}_8$ generates the direct summand $\Z_2$ of
$\pi^5_{23}$, $\nu_5\bar\varepsilon_8p'_{23}$ is not zero in $[M^{23},
 S^5]$. This implies $a=1$. 
%Therefore, we have 
%\[
% \Delta(\overline{P(E\theta)}+\overline{\nu_6\kappa_9}) =
% [\nu_5\bar\varepsilon_8]p'_{23} +
% b[2\iota_5] \zeta_5 \sigma_{16} p'_{23},
%\]
%where $b \in \{0,1,2,3\}$. 
\end{proof}

Thus, by applying $i_{G\ast}$ to (\ref{MM}), we have
$i_{G\ast}[\nu_5\bar\varepsilon_8]p'_{23}=0$. 
By (\ref{cohomotopy}), the relation $i_{G\ast}[\nu_5\bar\varepsilon_8]p'_{23}=0$ implies 
$i_{G\ast}[\nu_5\bar\varepsilon_8]$ is contained in the image of the map
$2\iota_{23}^{\ast} : \pi_{23}(G_2 : 2) \to \pi_{23}(G_2 :2)$. This and
Proposition 1.2 complete the proof of Theorem 1.3.

As an application of Proposition~\ref{G20} and Theorem~\ref{pi23G_2}, we show
\begin{prop}
$\pi_{23}(V_{7,2}:2) \cong (\Z_4)^2 \oplus (\Z_2)^2$. 
\end{prop}
To show this proposition, we will need some relations. 
lemma. 
\begin{lem}\label{poy}
\begin{itemize}
\item[(1)] $\mu' \in \{\eta_3, 2\iota_4, \mu_4\}_1 \bmod 2\mu'$ 
\item[(2)] $\bar\mu' \in \{\mu', 4\iota_{14}, 4\sigma_{14}\}_1  \bmod \nu'\mu_6\sigma_{15}$
\end{itemize}
\end{lem}
\begin{proof}
(1) The indeterminacy is $\eta_3 \circ E\pi^3_{13}+\pi^3_5 \circ
 \mu_5$. Here, $\pi^3_5 \circ \mu_5 =\{\eta^2_3\mu_5\}=\{2\mu'\}$ and
 $\eta_3 \circ E\pi^3_{13}=\eta_3 \circ E\{\eta_3\mu_4, \varepsilon'\}=\{2\mu', \eta_3E\varepsilon'\}$. By \cite[p.68]{T}, we have $- (\eta_3 \circ
 E\varepsilon')= \eta_3 \circ
 (-E\varepsilon') \in \eta_3 \circ (-E\{\nu', 2\nu_6, \nu_9\}) \subset
 \eta_3 \circ \{E\nu', 2\nu_7, \nu_{10}\}_1 = \{\eta_3, E\nu', 2\nu_7\}_1
 \circ \nu_{11} \supset \{\eta_3, E\nu', \nu_7\}_1 \circ 2\iota_{11} \circ
 \nu_{11} \ni \varepsilon_3 \circ 2\iota_{11} \circ \nu_{11}=0$. The
 indeterminacy is $\eta_3 \circ \pi^4_{11} \circ \nu_{11} =
 0$. Therefore, we have $\eta_3 \circ E\varepsilon'=0$, and the
 assertion is proved. 

(2) The indeterminacy is $\mu' \circ E\pi^{13}_{21} + \pi^3_{15} \circ
 4\sigma_{15}$. Since $\pi^3_{15} \cong (\Z_2)^2$~\cite[Theorem 7.6]{T},
 we have $\pi^3_{15} \circ 4\sigma_{15}=0$. 

We have $\mu' \circ E\pi^{13}_{21}=\{\mu'\eta_{14}\sigma_{15}, \mu'\bar\nu_{14}\}$. 
%We will show that $\mu' \circ E\pi^{13}_{21}=\{\nu'\mu_6\sigma_{15}, 2\bar\mu'\}$. 
Since $\mu'\nu_{14}=0$~\cite[Proposition 2.4 (1)]{Og} and
 $\bar\nu_{14}=\{\nu_{14}, \eta_{17}, \nu_{18}\}$~\cite[Lemma 6.2]{T},
 we have $\mu'\bar\nu_{14}=\mu' \circ \{\nu_{14}, \eta_{17},
 \nu_{18}\}=\{\mu', \nu_{14}, \eta_{17}\} \circ \nu_{19} \subset
 \pi^3_{19} \circ \nu_{19}$. 
By \cite[Theorem 12.6]{T}, $\pi^3_{19}=\{\mu_3\sigma_{12},
 \eta_3\bar\varepsilon_4\}$. 
%Since $\sigma_{11}\nu_{18}=P(\iota_{23})$,
We also have $\eta_3\bar\varepsilon_4=\bar\varepsilon_3\eta_{18}$ by (\ref{j7}),
 $\eta_{18}\nu_{21}=0$ and (\ref{rn}). %$\sigma_{12}\nu_{19}=0$. 
It follows that $\pi^3_{19} \circ \nu_{19}=0$,
 and hence $\mu'\bar\nu_{14}=0$. 

By \cite[Proposition 2.2 (4) ; the second]{Og}, we have
 $\mu'\eta_{14}=\nu'\mu_6$, so that
 $\mu'\eta_{14}\sigma_{15}=\nu'\mu_6\sigma_{15}$. 
%By \cite[Proposition 2.4 (1)]{Og}, 
%we have the equation
% $\mu'\varepsilon_{14} \in \mu' \circ \{\nu^2_{14}, 2\iota_{20},  \eta_{20}\}_1 =\{\mu', \nu^2_{14}, 2\iota_{20}\}_1 \circ \eta_{21}
% \subset \pi^3_{21} \circ \eta_{21}$. It suffices to calculate the generators
% of $\pi^3_{21} \circ \eta_{21}=\{\mu'\sigma_{14}\eta_{21},
% \eta_3\bar\mu_4\eta_{21}, \nu'\bar\varepsilon_6\eta_{21}\}$. We have
% $\nu'\bar\varepsilon_6\eta_{21}=\nu'\nu_6\sigma_9\nu^2_{16}=0$ by the
% relations
% $\bar\varepsilon_6\eta_{21}=\nu_6\sigma_9\nu^2_{16}$~\cite[Lemma
% 12.10]{T} and $\nu'\nu_6 \in \pi^3_9=0$, 
% $\mu'\sigma_{14}\eta_{21}=\mu'\eta_{14}\sigma_{15}=\nu'\mu_6\sigma_{15}$
% obtained by the relation $\mu'\eta_{14}=\nu'\mu_6 $ from \cite[Proposition 2.2 (4) ; the second]{Og} and $\eta_3\bar\mu_4\eta_{21}=\eta^2_3\bar\mu_5=2\bar\mu'$ 
\end{proof}

{\bf Proof of Proposition 4.9}

By using the fibering $S^3 \rarrow{i} G_2 \rarrow{p} V_{7,2}$,
 %$G_2/S^3=V_{7,2}$, 
we have an exact sequnece:
\[
 \cdots \to \pi^3_{23} \to \pi_{23}(G_2:2) \to \pi_{23}(V_{7,2}:2) \to \pi^3_{22} \to \cdots. 
\]
Here, $i_{\ast}\pi_{23}(G_2 : 2)=\{i_{\ast}\langle
 P(E\theta)+\nu_6\kappa_9 \rangle, i_{\ast}\langle \eta^2_6 \rangle
 \mu_8\}$. 
%and we have $i_* \langle P(E\theta)+\nu_6\kappa_9 \rangle =[P(E\theta)+\nu_6\kappa_9]$.
%We consider the other generators. 

Recall that the structures of the homotopy groups $\pi_{23}^3=\{\nu'\bar\mu_6, \nu'\eta_6\mu_7\sigma_{16}\} \cong
 (\Z_2)^2$ from \cite[p.45]{MT1} and $\pi^3_{22}=\{\bar\mu', \nu'\mu_6\sigma_{15}\}\cong \Z_4 \oplus \Z_2$ from \cite[Theorem 12.9]{T}, where we have $\bar\mu' \in \{\mu', 4\iota_{14}, 4\sigma_{14}\}_1$ and $2\bar\mu' = \eta^2_3\bar\mu_5$ by \cite[p.137, Lemma 12.4, Theorem 12.9]{T}. 

%Let \mbox{$\tilde{\eta}_n\in\pi_{n+2}(M^{n+1})\cong\Z_4$} 
%for $n\ge 3$ be a coextension of $\eta_n$. 
%$\pi_5(G_2:2)=0$, $\pi_6(G_2) \cong \Z_3$ and $\pi_7(G_2:2)=0$.
Let us denote the connecting map induced by the fibering $G_2
 /S^3=V_{7,2}$ by $\Delta : \pi_n(V_{7,2} : 2) \to \pi^3_{n-1}$. 
By \cite[p.132]{Mi2}, we obtain 
\begin{equation}\label{i1}
\pi_k(G_2:2)=0 \quad \text{for} \quad 4 \le k \le 7 . 
\end{equation}
Therefore, $\Delta: \pi_5(V_{7,2}:2) \to \pi^3_4$ is an isomorphism. 
Since $V_{7,2}$ is a $5$-sphere bundle over $S^6$, we have a cell
 structure : $V_{7, 2}=M^6 \cup e^{11}$. So, we obtain $\pi_5(V_{7,2} :2) \cong
 \pi_5(M^6 :2)=\{i'_6\} \cong \Z_2$ and $\Delta(i''i'_6)=\eta_3$, where $i'': M^6 \to V_{7,2}$ is the 
inclusion. We also have an isomorphism $\Delta: \pi_7(V_{7,2}:2) \to
\pi^3_6$. 
Since $\pi_7(V_{7,2}:2) \cong \pi_7(M^6:2)=\{\tilde{\eta}_5\} \cong \Z_4$, we
have $\Delta(i''\tilde{\eta}_5)=\nu'$. 

%We will see that the order of an element of 
% $\{\{i''i'_6,2\iota_5,\mu_5\},4\iota_{15},4\sigma_{15}\}_1$ is $4$. 
Since $(V_{7,2}, M^6)$ is $10$-connected and $[M^6, M^6]=\Z_4\{\iota_{M^6}\}$, 
%Since $[M^6, S^n]=0$ for $n=10, 11$, 
$i''_{\ast} : [M^6, M^6] \to [M^6, V_{7,2}]$ is an isomorphism. 
% by (\ref{ietap}). 
By (\ref{i1}), we have $[M^n, G_2]=0$ for
 $n=5, 6$. So, the fibering $G_2 / S^3=V_{7, 2}$ induces an isomorphism
 $\Delta : [M^6, V_{7,2}] \to [M^5, S^3]$. This and (\ref{hoco}) imply  
\begin{equation}\label{i2}
\Delta(i'') = \pm \bar\eta_3 . 
\end{equation}
Let $\tilde{\mu}_3 \in \pi_{13}(M^4)$ be a coextension of $\mu_3$. We
 set $\tilde{\mu}_n =E^{n-3} \tilde{\mu}_3 \in \pi_{n+10}(M^{n+1})$ for
 $n \ge 3$. By (\ref{ietap}), we obtain
\begin{equation}\label{i3}
2\tilde{\mu}_n = i'_{n+1}\eta_np'_{n+1}\tilde{\mu}_n = i'_{n+1}\eta_n\mu_{n+1} \quad
 \text{for} \quad n \ge 4 .
\end{equation}
By Lemma~\ref{poy} (1), 
%the definition of $\mu' \in \{\eta_3, 2\iota_4,
%\mu_4\}_1$~\cite[p.65]{T}, 
we have 
\begin{equation}\label{i4}
\Delta(i''\tilde{\mu}_5)=\pm \bar\eta_3\tilde{\mu}_4 =\pm \mu' .
\end{equation}
By (\ref{ietap}) and (\ref{i3}), we have $4i''\tilde{\mu}_5=0$. This
 implies $\sharp(i''\tilde{\mu}_5)=4$. 
Let $\mathfrak{m}$ be a representative of the Toda bracket $\{i''\tilde{\mu}_5,
 4\iota_{15}, 4\sigma_{22}\}_2$. Then, we have $4\mathfrak{m} \in
 4\{i''\tilde{\mu}_5, 4\iota_{15}, 4\sigma_{22}\}_2= \{i''\tilde{\mu}_5,
 4\iota_{15}, 4\sigma_{22}\}_2 \circ 4\iota_{30} = i''\tilde{\mu}_5
 \circ \{4\iota_{15}, 4\sigma_{15}, 4\iota_{22}\}_2 \subset
 i''\tilde{\mu}_5 \circ \{2\iota_{15}, 0,
 2\iota_{22}\}_1=i''\tilde{\mu}_5 \circ 2\pi^{15}_{23}=0$. 
We recall from \cite[p.137]{T} : 
%\begin{equation}\label{i5}
%\bar\mu' \in \{\mu', 4\iota_{14}, 4\sigma_{14}\}_1 .
%\end{equation}
By (\ref{i4}) and Lemma~\ref{poy}
%(\ref{i5})
, we have $\Delta\{i''\tilde{\mu}_5,
 4\iota_{15}, 4\sigma_{15}\}_2 \subset \{\pm \mu', 4\iota_{14},
 4\sigma_{14}\}_1 =\{(\pm \iota_3) \circ \mu', 4\iota_{14},
 4\sigma_{21}\}_1  \ni (\pm \iota_3) \circ \bar\mu'= \pm \bar\mu' \bmod
 \mu' \circ E\pi^{13}_{21} + \pi^3_{15} \circ 4\sigma_{15}=\{\nu'\mu_6\sigma_{15}\}$. Here, $(\pm
 \iota_3) \circ \bar\mu' = \pm \bar\mu'$ since $S^3$ is an $H$-space. 
%By Lemma~\ref{poy}, we have $\mu' \circ E\pi^{13}_{21}
%=\{\mu'\eta_{14}\sigma_{15}, \mu'\varepsilon_{14}\}=\{\nu'\mu_6\sigma_{15}\}$. 
%Since $\pi^3_{15} \circ 4\sigma_{15}=0$~\cite[Theorem 7.6]{T}, 
Therefore, $\Delta(\mathfrak{m}) \equiv \pm \bar\mu' \bmod \nu'\mu_6\sigma_{15}$. 

Hence, we have the split exact sequence:
%Note that $\bar{\mu}' \in \{\mu',4\iota_{14},4\sigma_{14}\}_1
%\equiv \{\{\eta_3,2\iota_4,\mu_4\},4\iota_{14},4\sigma_{14}\}_1$. 
\[
 0 \to \pi_{23}(G_2:2) \to \pi_{23}(V_{7,2} :2) \to \pi^3_{22} \to 0 ,  
\]
so that the splitting $\pi_{23}(V_{7,2} : 2)=i_{\ast}\pi_{23}(G_2 : 2)
 \oplus \Z_4\{\mathfrak{m}\} \oplus \Z_2\{i''\tilde{\eta}_5\mu_7\sigma_{16}\} \cong
 (\Z_4)^2 \oplus (\Z_2)^2$ gives the isomorphism in the proposition.

\noindent
Toshiyuki Miyauchi\\
Department of Applied Mathematics, Faculty of Science, Fukuoka 
University\\
8-19-1 Nanakuma, Jonan-ku, Fukuoka 814-0180, Japan\\
tmiyauchi@fukuoka-u.ac.jp\\

\noindent
Juno Mukai\\
Shinshu University\\
3-1-1 Asahi, Matsumoto, Nagano 390-8621, Japan\\
jmukai@shinshu-u.ac.jp\\

\noindent
Mariko Ohara\\
Faculty of Science, Shinshu University\\
3-1-1 Asahi, Matsumoto, Nagano 390-8621, Japan\\
primarydecomposition@gmail.com\\

\end{document}